\documentclass[12pt]{amsart}
\usepackage{amssymb}
\usepackage{amsmath}
\usepackage{amscd}
\usepackage{mathrsfs}
\usepackage[all]{xypic}
\usepackage{epsbox}

\newtheorem{thm}{Theorem}[section]
\newtheorem{lem}[thm]{Lemma}

\newtheorem{prop}[thm]{Proposition}

\newtheorem{fermatthm}{Theorem}[section]

\theoremstyle{definition}

\newtheorem{rem}[thm]{Remark}
\newtheorem{defn}[thm]{Definition}

\theoremstyle{remark}

\numberwithin{equation}{section}

\def\R{{\mathbb R}}
\def\Z{{\mathbb Z}}
\def\C{{\mathbb C}}

\def\G{\Gamma}

\allowdisplaybreaks[1]

\begin{document}
\title[A nontrivial algebraic cycle in
the Jacobian variety]
{A nontrivial algebraic cycle in
the Jacobian variety of the Fermat sextic}

\author[Yuuki Tadokoro]{Yuuki Tadokoro}
\address{Natural Science Education,
Kisarazu National College of Technology,
2-11-1 Kiyomidai-Higashi,
Kisarazu, Chiba 292-0041, Japan}
\email{tado@nebula.n.kisarazu.ac.jp}

\begin{abstract}
We compute some value of the harmonic volume for the Fermat sextic.
Using this computation, we prove that some special algebraic cycle
in the Jacobian variety of the Fermat sextic is not algebraically
equivalent to zero.
\end{abstract}

\maketitle
\section{Introduction}
B. Harris \cite{H-1} defined the harmonic volume for
the compact Riemann surface $X$
of genus $g \geq 3$,
using Chen's iterated integrals \cite{C}.
Let $J(X)$ be the Jacobian variety of $X$.
By the Abel-Jacobi map $X\to J(X)$, $X$ is
embedded in $J(X)$.
By a consideration of the special harmonic volume,
Harris \cite{H-3} proved that the algebraic cycle $F(4)-F(4)^{-}$
is not algebraically
equivalent to zero in $J(F(4))$.
Here, $F(4)$ is the Fermat quartic, which is a compact Riemann
surface of genus $3$.
Ceresa \cite{Ce} showed that
the algebraic cycle $X-X^{-}$ is not algebraically equivalent to zero
in $J(X)$
for a generic $X$.
We know few explicit nontrivial examples
except for $F(4)$.
Harris \cite{H-4} used the special feature of $F(4)$ that
its normalized period matrix has entries in a discrete subring of $\C$.
The Fermat sextic $F(6)$ has the same feature.
We use this and prove
\begin{fermatthm}\label{Fermat}
Let $F(6)$ be the Fermat sextic.
Then, the algebraic cycle $F(6)-F(6)^-$ is not algebraically
equivalent to zero in $J(F(6))$.
\end{fermatthm}
We compute
iterated integrals with some common base point of $F(6)$.
This is a similar computation
of Tretkoff and Tretkoff \cite{T-T}.
In order to compute the Poincar\'e dual of $F(6)$,
we use the result of Kamata \cite{K} for
the intersection number of the first integral homology class
of the Fermat curves.
It is difficult to
apply Harris' method to other Fermat curves.
We \cite{T-2} proved the same fact as the Klein quartic,
but we did not use the above special feature.

Now we describe the contents of this paper
briefly.
In \S \ref{Harmonic volumes and algebraic cycles},
we recall the definition and fundamental properties
of the harmonic volume and algebraic cycle in $J(X)$.
\S \ref{Iterated integrals of the Fermat curves}
is devoted to the computation of
iterated integrals of the Fermat curves.
In the latter half of this section,
we prove that iterated integrals on those curves
are represented by some special values of
the generalized hypergeometric function ${}_3F_2$.
It was introduced in \cite{T-2} but not proved.
In \S \ref{A nontrivial algebraic cycle in $J(F(6))$},
we prove Main Theorem,
using the numerical calculation by the MATHEMATICA program.

\noindent
{\bf Acknowledgments.}
The author is grateful to Nariya Kawazumi
for valuable advice and support.
This work is partially supported by 21st
Century COE program  (University of Tokyo) by
the Ministry of Education, Culture, Sports, Science and Technology.

\section{Harmonic volumes and algebraic cycles}
\label{Harmonic volumes and algebraic cycles}
Let $R$ be
a discrete subring of $\C$.
We suppose that all the entries
of the period matrix of the compact Riemann surface
$X$ can be reduced to elements of $R$.
Harris \cite{H-4} pointed out that
we may replace $\Z[\sqrt{-1}]$ for the Fermat quartic
in Harris' method in \cite{H-3}
with $R$.
We recall the harmonic volume for such $X$
as follows.
Let $H^{1,0}_R$ denote
the space of homolophic 1-forms on $X$ with $R$-periods.
It is a $g$-dimensional $\C$-vector space.
We choose
a basis $\{K_1,K_2,\ldots,K_{2g}\}$ of
the first integral homology group $H_1(X;\Z)$ of $X$.
\begin{defn}
[\cite{H-4}]
The harmonic volume is
defined to be the homomorphism $(H^{1,0}_R)^{\otimes_R 3}\to \C/{R}$ by
$$
I_R(\omega_1\otimes \omega_2\otimes \omega_3)
=\sum_{r=1}^{2g}a_r\int_{C_r}\omega_1\omega_2
 \quad \mathrm{mod} \ R.$$
Here
$\omega_1\otimes \omega_2\otimes \omega_3$
is an element of $(H^{1,0}_R)^{\otimes_R 3}$,
$C_r$ is a loop in $X$ at the fixed base point $x_0$
whose homology class is $K_r$,
and the Poincar\'e dual of $\omega_k$
is equal to $\displaystyle\sum_{r=1}^{2g}a_rK_r$
($a_r\in \C$).
The integral $\displaystyle \int_{C_r}\omega_1\omega_2$
is Chen's iterated integral \cite{C}, that is,
$\displaystyle \int_{C_r}\omega_1\omega_2
=\int_{0\leq t_1\leq t_2\leq 1}f_i(t_1)f_j(t_2)dt_1dt_2$
for $C_r^{\ast}\omega_i=f_i(t)dt,$ $i=1,2$,
where $t$ is the coordinate in the unit interval $[0,1]$.
\end{defn}

We remark that $I_R$
dose not depend on the choice of the base point $x_0$.
It is a modified version
of the original harmonic volume $I$.
See Harris \cite{H-1} for $I$.

Let $J=J(X)$ be the Jacobian variety of $X$.
By the Abel-Jacobi map $X\to J(X)$,
$X$ is embedded in $J(X)$.
The algebraic $1$-cycle $X-X^{-}$ in $J(X)$
is homologous to zero.
Here we denote by $X^{-}$ the image of $X$
under the multiplication map by $-1$.
We recall the relation between the harmonic volume and
algebraic 1-cycle $X-X^-$ in $J$.
We say the
algebraic cycle $X-X^{-}$ is
{\it algebraically equivalent to zero in} $J$
if there exists a topological $3$-chain
$W$ such that $\partial W=X-X^-$
and $W$ lies on $S$,
where $S$
is an algebraic (or complex analytic)
subset of $J$
of complex dimension $2$ (Harris \cite{H-4}).
The chain $W$ is unique up to
$3$-cycles.
Harris proved the key theorem.
\begin{thm}[Section 2.7 in \cite{H-4}]
\label{a sufficient condtion such that cycles are nontrivial}
If the algebraic cycle $X-X^{-}$ is
algebraically equivalent to zero in $J$,
then $2I_R(\omega)\equiv 0$ modulo $R$
for each $\omega\in (H^{1,0}_R)^{\otimes_R 3}$.
\end{thm}
See Harris \cite{H-3, H-4} for details.
In \S \ref{A nontrivial algebraic cycle in $J(F(6))$},
we find some element $\omega\in (H^{1,0}_R)^{\otimes_R 3}$
such that $2I_{R}(\omega)\not\equiv 0$ modulo $R$ for
the Fermat sextic.

\section{Iterated integrals of the Fermat curves}
\label{Iterated integrals of the Fermat curves}
In this section we compute iterated integrals of the
Fermat sextic.
Let $H^{1,0}$ denote 
the space of holomorphic 1-forms on $X$
We choose a basis $\{\omega_1,\omega_2,\ldots,\omega_g\}$
of $H^{1,0}$.
Let $\gamma$ be a loop in $X$ at some base point.
We remark that the iterated integral
$\displaystyle \int_{\gamma}\omega_i\omega_j$
depends on the choice of the base points
and is invariant under homotopy relative a fixed base point.
This iterated integral and the quadratic period defined by Gunning \cite{G}
are essentially same except for the sign.

For $N\in \Z_{\geq 3}$, let $F(N)=\{(X:Y:Z)\in\C P^2; X^N+Y^N=Z^N\}$
denote the Fermat curve of degree $N$,
which is a compact Riemann surface of genus $(N-1)(N-2)/{2}$.
Let $x$ and $y$ denote $X/{Z}$ and 
$Y/{Z}$ respectively. 
The equation $X^N+Y^N=Z^N$
induces $x^N+y^N=1$.
Using this coordinate $(x, y)\in F(N)$,
 the holomorphic map $\pi: F(N)\to\C P^1$ is defined by
 $\pi(x,y)=x$.
It is clear that $\pi$
is an $N$-sheeted covering $F(N) \to \C P^1$, 
branched over $N$ branch points 
$\{\zeta_N^i\}_{i=0,1,\ldots,N-1}\subset \C P^1$.
Here $\zeta_N$ denotes $\mathrm{exp}(2\pi\sqrt{-1}/{N})$.
Holomorphic automorphisms $\alpha$ and $\beta$ of $F(N)$
are defined by
$\alpha(X:Y:Z)=(\zeta_N X:Y:Z)$ and $\beta(X:Y:Z)=(X:\zeta_N Y:Z)$
respectively.
We have that $\alpha\beta=\beta\alpha$ and
 the subgroup of the holomorphic automorphisms of $F(N)$
which is generated
by $\alpha$ and $\beta$ is isomorphic to
$(\Z/{N\Z})\times (\Z/{N\Z})$.
Let $P_i$ and $Q_i$ denote $\alpha^i(1, 0)$
 and $\beta^i(0,1)$, $i=0,1,\ldots,N-1$ respectively.
We define a simply connected domain $\Omega$ by
$\C\setminus\bigcup_{j=0}^{j=N-1}
\{t\zeta^j;|t|\geq 1, t\in \R\}$.
Then $\pi^{-1}(\Omega)$ consists of
$N$ path-connected components and we denote
by $\Omega_i$ a connected
 component of $\pi^{-1}(\Omega)$ which contains $Q_i$,
$i=0,1,\ldots,N-1$.
Let $\gamma_0$ be a path
$[0,1]\ni t\mapsto 
(t,\sqrt[N]{1-t^N})\in F(N)$,
where $\sqrt[N]{1-t^N}$ is a real nonnegative analytic function
on $[0,1]$.
A loop in $F(N)$ is defined by
$$\kappa_0=\gamma_0\cdot(\beta\gamma_0)^{-1}\cdot
(\alpha \beta\gamma_0)\cdot(\alpha\gamma_0)^{-1},$$
where the product $\ell_1\cdot \ell_2$
indicates that we traverse $\ell_1$ first, then $\ell_2$.
We consider a loop $\alpha^i\beta^j\kappa_0$
as an element of the first homology group $H_1(F(N);\Z)$ of $F(N)$.
Kamata obtained the following lemma
for the intersection number of $H_1(F(N);\Z)$.
\begin{lem}[Section 5 in \cite{K}]
\label{intersection}
We have
$$
\left\{
\begin{array}{rcl}
(\kappa_0, \alpha\kappa_0)=& 1 &=-(\alpha\kappa_0, \kappa_0) \\
(\kappa_0, \beta\kappa_0)=& 1 &=-(\beta\kappa_0,\kappa_0) \\
(\kappa_0, \alpha\beta\kappa_0)=& -1 &
=-(\alpha\beta\kappa_0,\kappa_0) \\
(\kappa_0, \alpha\beta^{-1}\kappa_0)=& 0 &
=(\alpha\beta^{-1}\kappa_0,\kappa_0).
\end{array}
\right.
$$
\end{lem}

From this lemma, it is to show
\begin{prop}[Section 5 in \cite{K}]
\label{intersectin matrix}
We have $\{\alpha^i\beta^j\kappa_0\}_{i=0,1,\ldots,N-3,j=0,1,\ldots,N-2}$
is a basis of $H_1(F(N);\Z)$.
\end{prop}
\begin{rem}
Intersection matrix of
$\{\alpha^i\beta^j\kappa_0\}_{i=0,1,\ldots,N-3,j=0,1,\ldots,N-2}$
is given by $K$ in case (i) in
\cite{K}.
\end{rem}

It is a known fact
that $\{\omega^{\prime}_{r,s}=
x^{r-1}y^{s-1}dx/{y^{N-1}}\}_{r,s\geq1,r+s\leq N-1}$
is a basis of $H^{1,0}$ of $F(N)$.
It is clear that
$$\int_{\alpha^i\beta^j\gamma_0}\omega^{\prime}_{r,s}
=\zeta_N^{ir+js}\int_{\gamma_0}\omega^{\prime}_{r,s}
=\zeta_N^{ir+js}\frac{B(r/N,s/N)}{N}.$$
The integral of $\omega^{\prime}_{r,s}$
along $\alpha^i\beta^j\kappa_0$ is obtained as follows.
\begin{prop}
[Appendix in \cite{Gr}]
\label{Rohrlich}
We have
$$\int_{\alpha^i\beta^j\kappa_0}\omega^{\prime}_{r,s}
=B(r/N,s/N)(1-\zeta_N^r)(1-\zeta_N^s)\zeta_N^{ir+js}/{N}.$$
\end{prop}
We denote the $1$-form
$N\omega^{\prime}_{r,s}/{B^N_{r,s}}$
by
$\omega_{r,s}$.
Here, $B^N_{r,s}=B(r/N,s/N)$.
This implies
$\displaystyle \int_{\alpha^i\beta^j\kappa_0}\omega_{r,s}
\in \Z[\zeta_N]$.

Let $f_{r,s}$ 
be a real $1$-form on $[0,1]$ defined by
$\gamma_0^{\ast}\omega_{r,s}^{\prime}
=t^{r-1}\big(\sqrt[N]{1-t^N}\big)^{s-N}dt$
 for $r,s\geq1,r+s\leq N-1$.
The iterated integral
$\displaystyle
\int_{\gamma_0}\omega_{r,s}\omega_{l,m}=N^2\int_{\gamma}f_{r,s}f_{l,m}\Big/{(B^N_{r,s}B^{N}_{l,m})}$
is denoted by $x_{r,s,l,m}$.
Iterated integrals of $\omega_{r,s}$
 along the loop $\alpha^i\beta^j\kappa_0$
can be computed.
\begin{lem}\label{iterated integral of Fermat1}
We consider $\alpha^i\beta^j\kappa_0$ as a loop at the base point
$Q_j$.
Then the iterated integral 
$\displaystyle\int_{\alpha^i\beta^j\kappa_0}\omega_{r,s}\omega_{l,m}$
is given by
$$\zeta_N^{i(r+l)+j(s+m)}
\big\{(1-\zeta_N^{r+l})(1-\zeta_N^{s+m})
x_{r,s,l,m}
+(1-\zeta_N^s)(\zeta_N^{r+l}+\zeta_N^{l+m}
-\zeta_N^m-\zeta_N^l)\big\}.$$
\end{lem}
\begin{proof}
It is clear that
$\displaystyle\int_{\alpha^i\beta^j\kappa_0}\omega_{r,s}\omega_{l,m}
=\zeta_N^{i(r+l)+j(s+m)}\int_{\kappa_0}\omega_{r,s}\omega_{l,m}$.
We have only to compute 
$\displaystyle\int_{\kappa_0}\omega_{r,s}\omega_{l,m}$.
We denote
$\displaystyle
\left(\int_{\ell_1}+\int_{\ell_2}\right)\omega_{r,s}
\omega_{l,m}=
\int_{\ell_1}\omega_{r,s}
\omega_{l,m}+
\int_{\ell_2}\omega_{r,s}\omega_{l,m}$
only here.

Proposition \ref{Rohrlich},
the equation
$\displaystyle\int_{\gamma_0}\omega_{r,s}=1$, and
$$\int_{\gamma_0}\omega_{r,s}\omega_{l,m}
+\int_{\gamma_0^{-1}}\omega_{r,s}\omega_{l,m}
+\int_{\gamma_0}\omega_{r,s}\int_{\gamma_0^{-1}}\omega_{l,m}
=\int_{\gamma_0\cdot\gamma_0^{-1}}\omega_{r,s}\omega_{l,m}=0$$
give us the equation\\
$\displaystyle
\int_{\kappa_0}\omega_{r,s}\omega_{l,m}
=\left(\int_{\gamma_0}
+\int_{(\beta\gamma_0)^{-1}}
+\int_{\alpha\beta\gamma_0}
+\int_{(\alpha\gamma_0)^{-1}}\right)\omega_{r,s}\omega_{l,m}\\
+\int_{\gamma_0}\omega_{r,s}
\left(\int_{(\beta\gamma_0)^{-1}}+\int_{\alpha\beta\gamma_0}
+\int_{(\alpha\gamma_0)^{-1}}\right)\omega_{l,m}
+\int_{(\beta\gamma_0)^{-1}}\omega_{r,s}
\left(\int_{\alpha\beta\gamma_0}
+\int_{(\alpha\gamma_0)^{-1}}\right)\omega_{l,m}\\
\hspace*{250pt}+\int_{\alpha\beta\gamma_0}\omega_{r,s}
\int_{(\alpha\gamma_0)^{-1}}\omega_{l,m}\\
=\left(\int_{\gamma_0}
+\zeta_N^{s+m}\int_{\gamma_0^{-1}}
+\zeta_N^{r+s+l+m}\int_{\gamma_0}
+\zeta_N^{r+l}\int_{\gamma_0^{-1}}\right)\omega_{r,s}\omega_{l,m}\\
+\int_{\gamma_0}\omega_{r,s}
\left(-\zeta_N^{m}\int_{\gamma_0}+\zeta_N^{l+m}\int_{\gamma_0}
-\zeta_N^{l}\int_{\gamma_0}\right)\omega_{l,m}
-\zeta_N^{s}\int_{\gamma_0}\omega_{r,s}
\left(\zeta_N^{l+m}\int_{\gamma_0}
-\zeta_N^{l}\int_{\gamma_0}\right)\omega_{l,m}\\
\hspace*{250pt}-
\left(\zeta_N^{r+s}\int_{\gamma_0}\omega_{r,s}\right)
\zeta_N^{l}\int_{\gamma_0}\omega_{l,m}
\\
=\left\{(1+\zeta_N^{r+s+l+m})\int_{\gamma_0}
-(\zeta_N^{s+m}+\zeta_N^{r+l})\int_{\gamma_0}
\right\}\omega_{r,s}\omega_{l,m}+
(\zeta_N^{s+m}+\zeta_N^{r+l})\int_{\gamma_0}\omega_{r,s}
\int_{\gamma_0}\omega_{l,m}\\
\hspace*{200pt}
-\zeta_N^{m}+\zeta_N^{l+m}
-\zeta_N^{l}
-\zeta_N^{s+l+m}
+\zeta_N^{s+l}
-\zeta_N^{r+s+l}\\
=(1-\zeta_N^{r+l})(1-\zeta_N^{s+m})
\int_{\gamma_0}\omega_{r,s}\omega_{l,m}
+(1-\zeta_N^s)(\zeta_N^{r+l}+\zeta_N^{l+m}
-\zeta_N^m-\zeta_N^l).
$
\end{proof}

We define a path $\gamma_j$
by $\gamma_0\cdot(\beta^j\gamma_0)^{-1}$.
Let $\gamma_{i,j}$ denote the loop
$\gamma_j\cdot
(\alpha^i\beta^j\kappa_0)
\cdot\gamma_j^{-1}$.
Using the above lemma, we have
iterated integrals of $\omega_{r,s}$ along the loop
$\gamma_{i,j}$ at
the common base point $Q_0$.
\begin{thm}\label{iterated integral of Fermat2}
The iterated integral 
$\displaystyle\int_{\gamma_{i,j}}\omega_{r,s}\omega_{l,m}$
is given by
\begin{align*}
\zeta_N^{i(r+l)+j(s+m)}&
\big\{(1-\zeta_N^{l+r})(1-\zeta_N^{m+s})
x_{r,s,l,m}
+(1-\zeta_N^s)
(\zeta_N^{l+r}+\zeta_N^{l+m}-\zeta_N^m-\zeta_N^l)
\big\}\\
+&(1-\zeta_N^{js})(1-\zeta_N^l)(1-\zeta_N^m)\zeta_N^{il+jm}
-(1-\zeta_N^{jm})
(1-\zeta_N^r)(1-\zeta_N^s)\zeta_N^{ir+js}.
\end{align*}
\end{thm}
Tretkoff and Tretkoff \cite{T-T}
computed the quadratic periods with another
base point by similar computation.
\begin{proof}
We have
\begin{align*}
&\int_{\gamma_{i,j}}\omega_{r,s}\omega_{l,m}
=\int_{\alpha^i\beta^j\kappa_0}\omega_{r,s}\omega_{l,m}
+\int_{\gamma_j}\omega_{r,s}
\int_{\alpha^i\beta^j\kappa_0}\omega_{l,m}
-\int_{\alpha^i\beta^j\kappa_0}\omega_{r,s}
\int_{\gamma_j}\omega_{l,m}.
\end{align*}
From
this equation and Lemma \ref{iterated integral of Fermat1},
the result follows.
\end{proof}

For the numerical calculation of 
$\displaystyle x_{r,s,l,m}$,
we recall the generalized hypergeometric function ${}_3F_2$.
Let $\Gamma(\tau)$ denote 
the gamma function 
$\displaystyle\int_{0}^{\infty}e^{-t}t^{\tau-1}dt$
for $\tau>0$.
We define $(\alpha, n)$ by
$\G(\alpha +n)/{\G(\alpha)}$
for $n\in \Z_{\geq 0}$.
For $x\in\{z\in \C; |z|<1\}$
and $\alpha_1,\alpha_2,\alpha_3,\beta_1,\beta_2>-1$,
the generalized hypergeometric function ${}_3F_2$
is defined by
$${}_3F_2{\Big(}
\left.
\begin{array}{c}
\alpha_1,\alpha_2,\alpha_3\\
\beta_1,\beta_2
\end{array}
\right.
;x{\Big)}
=\sum_{n=0}^{\infty}{{(\alpha_1,n)(\alpha_2,n)(\alpha_3,n)}
\over{(\beta_1,n)(\beta_2,n)(1,n)}}x^n.
$$
\begin{prop}\label{hypergeometric function}
Let $\Delta$ be a $1$-simplex
$\{(u,v)\in\R^2;0\leq v\leq 1,
0\leq u\leq v\}$.
If $a,b,p,q>0, b<1$, then we have\\
$\displaystyle
\int_{\Delta} u^{a-1}(1-u)^{b-1}
v^{p-1}(1-v)^{q-1}dudv
={{B(a+p,q)}\over{a}}
\lim_{
\begin{subarray}{c}
t\to 1-0\\
t\in \R
\end{subarray}
}{}_3F_2{\Big(}
\left.
\begin{array}{c}
a,1-b,a+p\\
1+a,a+p+q
\end{array}
\right.
;t{\Big)}$.
\end{prop}
\begin{proof}
Using the equation
\begin{align*}
\int_{0}^{v}u^{a-1}(1-u)^{b-1}du
=&\int_{0}^{v}\sum_{n=0}^{\infty}u^{a-1}
\left(
\begin{array}{c}
b-1\\
n
\end{array}
\right)
(-u)^ndu\\
=&\sum_{n=0}^{\infty}
{{(1-b,n)}\over{(1,n)}}\int_{0}^{v}
u^{n+a-1}du,
\end{align*}
we compute as follows:
\begin{align*}
&
\int_{0}^{1}v^{p-1}(1-v)^{q-1}
\int_{0}^{v}u^{a-1}(1-u)^{b-1}du
dv\\
=&
\int_{0}^{1}v^{p-1}(1-v)^{q-1}
\sum_{n=0}^{\infty}
{{(1-b,n)}\over{(1,n)}}\int_{0}^{v}
u^{n+a-1}du\\
=&\sum_{n=0}^{\infty}
\int_{0}^{1}v^{a+p+n-1}(1-v)^{q-1}
{{(1-b,n)}\over{(1,n)}}
{{1}\over{a+n}}dv\\
=&\sum_{n=0}^{\infty}B(a+p+n,q){{(1-b,n)}\over{(1,n)}}
{{1}\over{a+n}}\\
=&
\sum_{n=0}^{\infty}{{\G(a+p+n)\G(q)}
\over{\G(a+p+q+n)}}{{(1-b,n)}\over{(1,n)}}
{{1}\over{a+n}}\\
=&
{{\G(a+p)\G(q)}\over{a\G(a+p+q)}}
\sum_{n=0}^{\infty}{{a}\over{a+n}}{{\G(a+p+q)}\over{\G(a+p+q+n)}}
{{\G(a+p+n)}\over{\G(a+p)}}{{(1-b,n)}\over{(1,n)}}\\
=&
{{B(a+p,q)}\over{a}}
\lim_{
\begin{subarray}{c}
t\to 1-0\\
t\in \R
\end{subarray}
}{}_3F_2{\Big(}
\left.
\begin{array}{c}
a,1-b,a+p\\
1+a,a+p+q
\end{array}
\right.
;t{\Big)}.
\end{align*}
\end{proof}

From this proposition,
we have
\begin{lem}\label{iterated integral computation of Fermat}
$$
x_{r,s,l,m}=
\frac{\displaystyle N^2\int_{\gamma}f_{r,s}f_{l,m}}{B^N_{r,s}B^{N}_{l,m}}=
\frac{NB^N_{r+l,m}}{rB^N_{r,s}B^{N}_{l,m}}
\lim_{
\begin{subarray}{c}
t\to 1-0\\
t\in \R
\end{subarray}
}{}_3F_2{\Big(}
\left.
\begin{array}{c}
r/{N},1-s/{N},(r+l)/{N}\\
1+r/{N},(r+l+m)/{N}
\end{array}
\right.
;t{\Big)}.$$
\end{lem}

\section{A nontrivial algebraic cycle in $J(F(6))$}
\label{A nontrivial algebraic cycle in $J(F(6))$}
In this section, we consider only the case $N=6$.
We compute some value of the harmonic volume for the Fermat sextic $F(6)$.
This tells the nontriviality of the algebraic cycle $F(6)-F(6)^-$
in $J(F(6))$.
We have the genus of $F(6)$ is equal to $10$
and
$\{\omega_{r,s}\}_{r,s\geq1,r+s\leq 5}$
is a basis of $H^{1,0}$ of $F(6)$.
For the rest of this paper, we denote $\zeta=\zeta_6$ and $R=\Z[\zeta]$.
Proposition \ref{intersectin matrix}
gives that a set of loops
$\{\gamma_{0,0},\gamma_{0,1},\ldots,\gamma_{0,4},
\gamma_{1,0},\gamma_{1,1},\ldots,\gamma_{1,4},
\gamma_{2,0}
\ldots,
\gamma_{3,0},\gamma_{3,1},\ldots,\gamma_{3,4}\}$
may be considered as a basis of
the integral homology group $H_1(F(6);\Z)$ of $F(6)$.
Let $\mathrm{P.D.}\colon H^1(F(6);\C)\to H_1(F(6);\C)$
be the Poincar\'e dual.
\begin{lem}\label{PD}
Let $L_{i,k}$ be a linear combination
$\sum_{n=0}^{5}\zeta^{nk}\gamma_{i,n}$ in $H_1(F(6);\C)$.
Then we have
$$
\mathrm{P.D.}(\omega_{1,1})=
\frac{1}{122}\big\{
(60-13\zeta)L_{0,1}
-(15-49\zeta)L_{1,1}
-(43-51\zeta)L_{2,1}
-(50-21\zeta)L_{3,1}
\big\}.
$$
\end{lem}
\begin{proof}
Since $\beta_{\ast}(\gamma_{i,j})=\gamma_{i,j+1}$
as a homology class,
we obtain 
$$\beta_{\ast}L_{i,k}=\zeta^{-k}L_{i,k}.$$
We have
$$
\beta_{\ast}(\mathrm{P.D.}(\omega_{1,1}))
=\mathrm{P.D.}((\beta^{-1})^{\ast}\omega_{1,1})
=\zeta^5\mathrm{P.D.}(\omega_{1,1}).
$$
Since $\beta_{\ast}L_{i,1}=\zeta^5L_{i,1}$,
there exist constants $\lambda_0,\ldots,\lambda_3\in \C$ such that
$\mathrm{P.D.}(\omega_{1,1})=\sum_{i=0}^{3}\lambda_iL_{i,1}$.
The result follows from Proposition \ref{Rohrlich}
and the equations
\begin{align*}
\int_{\gamma_{0,0}}\omega_{1,1}&=(\mathrm{P.D.}(\omega_{1,1}),\gamma_{0,0})
&=\sum_{i=0}^{3}\lambda_i(L_{i,1},\gamma_{0,0})
&=(\zeta^5-\zeta)\lambda_0+(\zeta-1)\lambda_1,\\
\int_{\gamma_{1,0}}\omega_{1,1}&=(\mathrm{P.D.}(\omega_{1,1}),\gamma_{1,0})
&=\sum_{i=0}^{3}\lambda_i(L_{i,1},\gamma_{1,0})
&=(1-\zeta)\lambda_0+(\zeta^5-\zeta)\lambda_1+(\zeta-1)\lambda_2,\\
\int_{\gamma_{2,0}}\omega_{1,1}&=(\mathrm{P.D.}(\omega_{1,1}),\gamma_{2,0})
&=\sum_{i=0}^{3}\lambda_i(L_{i,1},\gamma_{2,0})
&=(1-\zeta)\lambda_1+(\zeta^5-\zeta)\lambda_2+(\zeta-1)\lambda_3,\\
\int_{\gamma_{3,0}}\omega_{1,1}&=(\mathrm{P.D.}(\omega_{1,1}),\gamma_{3,0})
&=\sum_{i=0}^{3}\lambda_i(L_{i,1},\gamma_{3,0})
&=(1-\zeta)\lambda_2+(\zeta^5-\zeta)\lambda_3.
\end{align*}
\end{proof}

Let
$\displaystyle \int_{L_{i,k}}\omega_{r,s}\omega_{l,m}$
denote
$\displaystyle \sum_{n=0}^{5}\zeta^{nk}
\int_{\gamma_{i,n}}\omega_{r,s}\omega_{l,m}$.
\begin{lem}\label{an iterated integral}
For $i=0,\ldots,3$, we have
$$
\int_{L_{i,1}}\omega_{1,2}\omega_{1,3}=
6\left\{\zeta^{2i}(1+\zeta)(x_{1,2,1,3}-1)-\zeta^{i}\right\}.
$$
\end{lem}
\begin{proof}
By Theorem \ref{iterated integral of Fermat2},
it is easy to compute
$$\int_{\gamma_{i,n}}\omega_{1,2}\omega_{1,3}=
\zeta^{2i-1}(1+\zeta)(x_{1,2,1,3}-1)
+2(1-\zeta^{2n})\zeta^{i+3n}(1-\zeta)
-(1-\zeta^{3n})\zeta^{i+2n}(1-2\zeta).
$$
Using this equation,
we obtain the result in a straightforward way.
\end{proof}

\begin{thm}\label{Fermat}
Let $F(6)$ be the Fermat sextic.
Then, the cycle $F(6)-F(6)^-$ is not algebraically
equivalent to zero in $J(F(6))$.
\end{thm}

\begin{proof}
By the definition of the harmonic volume $I_{R}$,
we have
$$
I_{R}(\omega_{1,2}\otimes \omega_{1,3}\otimes\omega_{1,1})
\equiv
\sum_{i=0}^{3}\lambda_i\int_{L_{i,1}}\omega_{1,2}\omega_{1,3}
\ \mathrm{mod}\ R.
$$
Using Lemma \ref{PD} and \ref{an iterated integral},
we obtain
$$2I_{R}(\omega_{1,2}\otimes \omega_{1,3}\otimes\omega_{1,1})
\equiv
\frac{6}{61}
\big\{
(42-3\zeta)x_{1,2,1,3}
-95+46\zeta
\big\}
\ \mathrm{mod}\ R,$$
and denote it by $\alpha$.
By Lemma \ref{iterated integral computation of Fermat}
and the numerical calculation (Figure \ref{Fermat-program}
in Appendix),
we obtain the value
$$2\Re (\alpha)
\equiv
\frac{6}{61}
(
81x_{1,2,1,3}-144
)
\equiv
0.74286\pm 1\times 10^{-5}
\ \mathrm{mod}\ \Z.$$
The result follows from
Theorem \ref{a sufficient condtion such that cycles are nontrivial}
and the lemma
$$2\Re(\alpha)\not\in\Z\Rightarrow \alpha\not\in \Z[\zeta].$$
\end{proof}

\section{Appendix}
We introduce the MATHEMATICA program \cite{W} in
the proof of Theorem \ref{Fermat}.
\begin{figure}[h]
\begin{center}
\psbox[width=15cm]{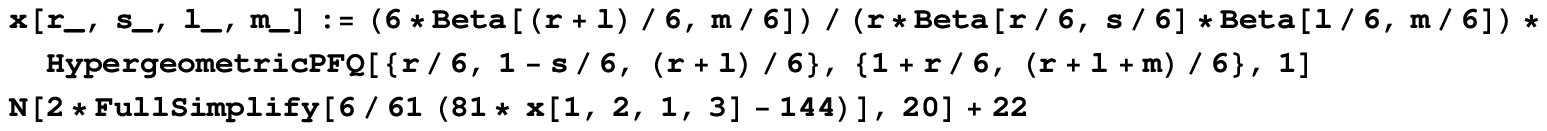}
\caption{A numerical calculation program 
in the proof of Theorem \ref{Fermat}}
\label{Fermat-program}
\end{center}
\end{figure}

\end{document}